\documentclass[reqno]{amsart}
\usepackage[utf8]{inputenc}
\usepackage{tikz-cd}
\usepackage[margin=1.3in]{geometry}  
\usepackage{graphicx}    
\usepackage{graphicx}
\usepackage{float}
\usepackage{dsfont}
\usetikzlibrary{matrix}
\newcommand*\Z{\mathbb{Z}}
\newcommand*\ZZZ{|[draw,circle]| \QQ}

\usepackage{amsmath}               
\usepackage{amsfonts}              
\usepackage{amsthm}                
\usepackage{enumitem}
\usepackage{amssymb}
\usepackage{tabularx}
\usepackage{amsmath}
\usepackage{mathtools}
\usepackage[mathscr]{euscript}
\usepackage[utf8]{inputenc}
\usepackage[english]{babel}
\usepackage{multicol}
\usepackage{hyperref}
\usepackage{tikz}
\usetikzlibrary{matrix}

\tikzset{diagram/.style={matrix of math nodes, inner sep=0pt, row
    sep=#1, column sep=2.5em, text height=1.5ex, text depth=.25ex,
    nodes={inner sep=1ex}}}
\tikzset{diagram/.default=2.5em}
\newcommand\diagram{\path node[diagram]}

\newtheorem{thm}{Theorem}[section]
\newtheorem{prop}[thm]{Proposition}
\newtheorem{lemma}[thm]{Lemma}
\newtheorem{cor}[thm]{Corollary}

\theoremstyle{definition}
\newtheorem{defn}{Definition}[section]

\newtheorem{rmk}{Remark}[section]
\newtheorem{claim}{Claim}
\newtheorem*{related works}{Related Works}

\newtheorem*{question*}{Question}

\theoremstyle{definition}
\newtheorem{question}{Question}

\numberwithin{equation}{section}

\newcommand{\ld}{\lambda}
\newcommand{\PP}{\CC{{P}}}
\newcommand{\NN}{\mathbb{N}}
\newcommand{\CC}{\mathbb{C}}

\newcommand{\QQ}{\mathbb{Q}}
\newcommand{\ZZ}{\mathbb{Z}}
\newcommand{\fq}{\mathbb{F}_q}
\newcommand{\ff}{\mathbb{F}}

\newcommand{\ir}{\mathrm{Irr}}
\newcommand{\ird}{\ir_{d,n}(\CC)}
\newcommand{\irdd}{\ir_{d,n+1}(\CC)}
\newcommand{\irdf}{\ir_{d,n}(\fq)}
\newcommand{\ii}{\ir_{d}(\CC)}

\newcommand{\red}{\mathrm{Red}}
\newcommand{\rd}{\red_{d,n}(\CC)}
\newcommand{\rdd}{\red_{d,n+1}(\CC)}
\newcommand{\sym}{\mathrm{Sym}}
\newcommand{\pl}{\mathrm{Poly}}
\newcommand{\pld}{\pl_{d,n}(\CC)}

\newcommand{\s}{{T}_{\ld,n}}
\newcommand{\sn}{{T}_{\ld,n+1}}
\newcommand{\rdn}{\red_{d,n+1}(\CC)}

\def\mul#1#2{\ensuremath{\left(\kern-.3em\left(\genfrac{}{}{0pt}{}{#1}{#2}\right)\kern-.3em\right)}}

\title[Cohomology of the space of complex irreducible polynomials in several variables]{Stability of the cohomology of the space of complex irreducible polynomials in several variables}

\author{Weiyan Chen}

\begin{document}

\maketitle

\begin{abstract}
    We prove that the space of complex irreducible polynomials of degree $d$ in $n$ variables satisfies two forms of homological stability: first, its cohomology stabilizes as $d\to\infty$, and second, its compactly supported cohomology stabilizes as $n\to\infty$. Our topological results are inspired by counting results over finite fields due to Carlitz and Hyde.
\end{abstract}

\section{Introduction}

The interests of counting irreducible polynomials over finite fields have stretched from eighteenth century to the modern era. Let $\irdf$ denote the number of irreducible polynomials of total degree $d$  in $n$ variables with coefficients in $\fq$ up to scalar multiplications.  For example, Gauss (\cite{Gauss}, page 611) first calculated the size of $\ir_{d,1}(\fq)$. In 1963, Carlitz \cite{Car} proved that for integers $n>1$
\begin{equation}\label{Car}
    \frac{|\irdf|}{q^{{{d+n}\choose n}-1}}\longrightarrow 1+q^{-1}+q^{-2}+\cdots \ \ \ \ \ \ \ \  \text{as $d\to\infty$.}
\end{equation}
In 2018,  Hyde (Theorem 1.1 in \cite{Hyde}) proved that $|\irdf|$ is always a polynomial in $q$ which converges coefficient-wise to a formal power series $P_d(q)$ as $n\to\infty$. In other words, in the formal power series ring $\QQ[[q]]$ equipped with the $q$-adic topology (under which higher powers of $q$ are considered smaller),
\begin{equation}\label{Hyde}
    |\irdf|\longrightarrow P_d(q) \ \ \ \ \ \ \ \ \ \ \ \ \ \ \ \ \ \ \ \ \ \ \ \  \text{as $n\to\infty$.}
\end{equation}

In this paper, we will pass from $\fq$ to $\CC$  and study the topology of the following manifold: 
$$\ird:=\{\text{irreducible complex polynomials in $n$ variables with degree $d$}\}/\CC^\times.$$ 
In  \cite{CEF}, Church-Ellenberg-Farb used the Grothendieck-Lefschetz trace formula  to connect asymptotic point-counts over finite fields and stability phenomena in cohomology. Heuristics based on this connection lead us to ask the following topological questions inspired by the aforementioned counting results of Carlitz and Hyde (see Section \ref{why} for a brief explanation of the heuristics):
\begin{question}\label{q1}
Does $H^i(\ird;\QQ)$ stabilize as $d\to\infty$?
\end{question}
\begin{question}\label{q2}
Does $H^i_c(\ird;\QQ)$ stabilize as $n\to\infty$?
\end{question}
Observe that $H^i_c(\ird;\QQ)$ is Poincar\'e dual to $H^{D-i}_c(\ird;\QQ)$ where $D$ is the real dimension of the manifold $\ird$. Thus, Question \ref{q2} equivalently asks if $\ird$ satisfies cohomological stability in codimensions. 

We will prove the following two theorems, each respectively answering the questions above affirmatively. 

\begin{thm}\label{homo stab low}
For $n>1$ and $d$ any positive integer, when $i\le2\bigg[{{d+n-1}\choose {n-1}}-n-1\bigg]$, we have
\[
H_i(\ird,\Z)\cong
  \begin{cases}
   \Z \ \ \ &  i\text{ is even} \\
   0 \ \ \ &  i\text{ is odd.}
  \end{cases}
\]
\end{thm}

\begin{rmk}

Theorem \ref{homo stab low} implies that $H^i(\ird,\Z)$ stabilizes as either $n$ or $d$ increases, giving a positive answer to Question \ref{q1}. In fact, Carlitz also proved that the same limit in (\ref{Car}) holds as $n\to\infty$ (see equation (11) in \cite{Car}), although he didn't state it in the main theorem. Thus, Theorem \ref{homo stab low} can be viewed as a topological analog of Carlitz' result. 
\end{rmk}



Observe that there is a natural inclusion $\ird \hookrightarrow\irdd$ given by forgetting the $(n+1)$-th variable. This inclusion is an embedding of a closed subspace, and hence a proper map. 

\begin{thm}
\label{high stab}
For $n,d>1$ and for any $i<\frac{2n}{d-1}-\frac{(d-2)(d-3)}{2}-1$, the natural inclusion $\ird \hookrightarrow\irdd$ induces an isomorphism
$$H^i_c(\ird; \QQ)\xleftarrow[]{\cong} H^i_c(\irdd;\QQ).$$
\end{thm}

$\ird$ is a complex manifold and satisfies Poincar\'e duality for compactly supported cohomology.  Theorem \ref{high stab} equivalently says that the cohomology of $\ird$ stabilizes in fixed codimensions as $n\to\infty$. Hence, Theorem \ref{homo stab low} and Theorem \ref{high stab} cover different ranges of the cohomology of $\ird$.

Unlike Theorem \ref{homo stab low}, Theorem \ref{high stab} only shows cohomological stability without telling us what the stable cohomology is. In the last part of the paper, we will study the limit
$$b_i(d):=\lim_{n\to\infty}\dim H^i_c(\ird;\QQ).$$
We prove that $b_i(d)=0$ when $i\le 2d$ and $d\ge2$ (Corollary \ref{van limit}). However, $b_i(d)$ are generally nonzero when $i$ is large enough. As examples, in the Appendix we compute $b_i(d)$ for all $i$ in the range $d\le 3$ and showed that $b_{11}(4)=1$. 

Our methods are topological and do not use the Grothendieck-Lefschetz trace formula. We will consider a stratification of the space of polynomials according to how they factor (Section \ref{pre}), and  then analyze the spectral sequence induced by the stratification (Section 4, 5 and 6).

\begin{rmk}[\textbf{Related works}]
Hyde (Theorem 1.22 in \cite{Hyde2}) recently proved that the compactly supported Euler characteristic of $\ird$ is 0 when $d>1$. Note that $\chi_c(\ird)=\chi(\ird)$ by Poincar\'e duality.  Since the stable cohomology of $\ird$ as in Theorem \ref{homo stab low} is supported in even degrees, we expect $\ird$ to have  nonzero odd cohomology groups in the unstable range.

Tommasi \cite{Tommasi} proved that the rational cohomology of the space $X_{d,n}$ of nonsingular complex homogeneous polynomials of degree $d$ in $n+1$ variables stabilizes as $d\to\infty$. Since the defining equation of any nonsingular hypersurface is irreducible,  $\ir_{d,n+1}(\CC)$ contains the projectivized $X_{d,n}/\CC^\times$. 
Comparing Theorem \ref{homo stab low} and Tommasi's result, we see that even though $\ird$ and $X_{d,n}/\CC^\times$ both satisfy cohomological stability as $d\to\infty$, their stable cohomology groups are different: Tommasi's theorem implies that the stable cohomology of $X_{d,n}/\CC^\times$ is isomorphic to the cohomology of $\mathrm{PGL}_{n+1}(\CC)$ which is generated by classes with odd degrees; in contrast, Theorem \ref{homo stab low} tells us that the stable cohomology of $\ird$ is supported in even degrees.


The theme of this paper is close to that of Farb-Wolfson-Wood \cite{FWW}, where they proved surprising coincidences in the Poincar\'e series of certain apparently unrelated spaces, which were predicted by the corresponding point-counting results over finite fields (Theorem 1.2 in \cite{FWW}). Our Theorem \ref{homo stab low} and \ref{high stab}, as well as the reasoning that leads us to discover them, provide another example where the Grothendieck-Lefschetz trace formula, despite not playing any role in the proofs, can still provide heuristics leading to plausible conjectures, which are then settled by topological methods. 
\end{rmk}

\section*{Acknowledgement}
The author would like to thank Ronno Das, Nir Gadish, and Trevor Hyde for helpful conversations, and thank an anonymous referee for pointing out an error in an earlier version of the paper.

\section{From counting to cohomology}
\label{why}
We will  briefly explain the heuristic that leads us to ask Question \ref{q1} and \ref{q2} from Carlitz' and Hyde's counting results. Our reasoning here was inspired by the  work of Church-Ellenberg-Farb \cite{CEF}.

For $X$ a variety over $\ZZ$, the Grothendieck-Lefschetz trace formula gives
\begin{equation}\label{GL}
    |X(\fq)| = \sum_i (-1)^i \mathrm{Trace}\Big(\mathrm{Frob}_q:\  H^i_\text{\'et,c}(X_{/{\overline{\fq}}};\QQ_{\ell})\Big)
\end{equation}
where $X(\fq)$ is the set of $\fq$-points on $X$, and the right hand side involves the trace of Frobenius acting on the compactly supported \'etale cohomology of $X$ over $\overline{\fq}$ with $\QQ_\ell$-coefficient for $\ell$ a prime not dividing $q$. Deligne proved that all the eigenvalues of Frobenius on $H^i_\text{\'et,c}(X;\QQ_{\ell})$ have absolute values no more than $q^{i/2}$ (Th\'eor\`eme 2 in \cite{Deligne}). 

For the sake of heuristic reasoning, let us suppose that there is a variety $X_{d,n}$ over $\ZZ$ such that $X_{d,n}(\fq)=\irdf$ and $X_{d,n}(\CC)=\ird$. Since Hyde proved that $|\irdf|$ is a polynomial in $q$, the Grothendieck-Lefschetz trace formula (\ref{GL}) together with Deligne's bounds would tell us that roughly the low $q$-powers in the polynomial $|\irdf|$ come from $H^i_c(\ird;\QQ)$ for $i$ small. Since Hyde (\ref{Hyde}) proved that the low $q$-powers in $|\irdf|$ converge as $n\to\infty$, one would expect that $H^i_c(\ird;\QQ)$ should stabilize as $n$ increases. Similarly, Carlitz (\ref{Car}) proved that the high $q$-powers in $|\irdf|$ converge as $d\to\infty$. One would therefore expect that $H^i(\ird;\QQ)$ should stabilize as $d$ increases, after applying Poincar\'e duality. These are the reasons why we ask Question \ref{q1} and \ref{q2} and expect positive answers.

\begin{rmk}[\textbf{Counting geometrically irreducible polynomials}]
It turns out that the variety $X_{d,n}$ satisfying our assumptions above does not exist. However, there does exist a variety $Y_{d,n}$ over $\ZZ$ such that $Y_{d,n}(\CC)=\ird$ and $Y_{d,n}(\fq)$ is the set of \emph{geometrically} irreducible polynomials, namely, polynomials over $\fq$ that cannot be written as a nontrivial product of polynomials over $\overline{\fq}$. Moreover, $|Y_{d,n}(\fq)|$ can be expressed in terms of $|\ir_{d/e,n}(\ff_{q^e})|$ for $e$ divisors of $d$. Hyde (personal communication) verified that $|Y_{d,n}(\fq)|$  satisfies the same convergence phenomena as $|\irdf|$. Therefore, one can make the heuristic reasoning above a rigorous argument if one replaces $|\irdf|$ by $|Y_{d,n}(\fq)|$, although we will not adopt this approach in the present paper.
\end{rmk}

\section{Preliminary lemmas}
\label{pre}
We first prove some preliminary results that will be used later in the paper. The results we collect here can be viewed as topological analogs of Lemma 2.1 in \cite{Hyde}.

Consider the space 
$$\pl_{\le d,n}(\CC):=\{\text{nonzero complex polynomials in $n$ variables with total degree $\le d$}\}/\CC^\times.$$
Note that $\pl_{\le d,n}(\CC)= \PP^{{d+n\choose n}-1}$ because there are $d+n\choose n$ many monomials of degree $\le d$ in $n$ variables. Next define $\pld:=\pl_{\le d,n}(\CC)\setminus \pl_{\le d-1,n}(\CC)$. This is the space of normalized multivariate polynomials with total degree $d$. 
\begin{lemma}\label{projective}
For any $d$ and $n$, we have a homeomorphism:
$$\pld\cong \CC^{\binom{d+n-1}{n}}\times \PP^{\binom{d+n-1}{n-1}-1}.$$
Thus, $\pld$ is homotopy equivalent to $\PP^{\binom{d+n-1}{n-1}-1}$.
\end{lemma}
\begin{proof}
Observe that any $f\in\pld$ can be written uniquely as 
$$f = f_d+f_{<d}$$
where $f_d$ is a homogeneous polynomial of degree $d$ (up to scalar) and $f_{<d}$ is an arbitrary polynomial of degree $<d$. The map $f\mapsto (f_{<d},f_d)$ gives the isomorphism.
\end{proof}

Define $$\rd:=\{f\in\pld:\text{$f$ is reducible}\}$$
which is a closed subspace of $\pld$ with open complement $\ird$. We have a long exact sequence:
\begin{equation}\label{les}
    \cdots \to H^i_c(\ird;\Z)\to H^i_c(\pld;\Z)\to H^i_c(\rd;\Z)\to \cdots
\end{equation}

Every $f\in\pld$ can be factorized uniquely into a product of  irreducible polynomials up to scalars $f=f_1f_2\cdots f_l$, which gives a unique partition $\ld_f$ of the integer $d=\deg f$ by
$$\ld_f: \deg(f_1)+\deg(f_2)+\cdots +\deg(f_l)=d.$$
For any partition $\ld$ of $d$ (written as $\ld\vdash d$ in the future), we define the following subspace of $\pld$:
$$\s:=\{f\in\pld : \ld_f=\ld\}.$$
We use $\sym^m X$ to denote the \emph{$m$-th symmetric power} of a topological space $X$. So $\sym^m X:=X^m/{S}_m$  where the symmetric group ${S}_m$ acts on $X^m$ by permuting the coordinates.  For $\ld\vdash d$ and for $j\in\Z_{>0}$, we will let $m_j(\ld)$ denote the multiplicity of $j$ in $\ld$. Every polynomial $f\in \s$ can be factorized uniquely into $\prod_{j=1}^d f_{j,1}f_{j,2}\cdots f_{j,m_j}$ where each $f_{j,k}\in \ir_{j,n}(\CC)$, up to reordering. Thus, we have
\begin{equation}\label{sym}
    \s\cong\prod_{j=1}^d\mathrm{Sym}^{m_j(\ld)}\Big(\ir_{j,n}(\CC)\Big).
\end{equation}
The unique factorization of  polynomials gives the following decomposition of $\pld$ into disjoint subsets:
\begin{equation}\label{strata}
    \pld = \bigcup_{\ld\vdash d} \s
\end{equation}
Let $(d)$ denote the trivial partition with a single part. Notice that $T_{(d),n}=\ird$.

We will focus on the decomposition of the space of reducible polynomials:
\begin{equation}\label{red strata}
    \rd = \bigcup_{\ld\vdash d, |\ld|\ge2} \s
\end{equation}
where $|\ld|:=\sum_j m_j(\ld)$ denote the total number of parts in the partition $\ld$.

\begin{lemma}\label{ss}
There is a spectral sequence 
\begin{equation}\label{spec}
    E_1^{p,q} = \bigoplus_{\ld:\ \ld\vdash d, |\ld|=d-p\ge2} H_c^{p+q}(\s;\Z)\ \Longrightarrow\ H^{p+q}_c(\rd;\Z).
\end{equation}
Moreover, its convergence happens at $E_{d-1}=E_\infty.$
\end{lemma}

\begin{proof}
Consider the following increasing filtration of $\rd$:
\begin{equation}\label{fil}
    \emptyset=\mathcal{F}_{0}\subset\mathcal{F}_{1}\subset\cdots\subset\mathcal{F}_{d-1}=\rd\ \ \ \ \text{ where each }\ \ \ \  \mathcal{F}_p := \bigcup_{\ld:\ |\ld|\ge d+1-p}\s.
\end{equation}
We claim that each $\mathcal{F}_p$ is a closed subspace of $\pld$. In fact, we have $\s\subseteq \overline{T_{\mu,n}}$ if  $\ld$ is finer than or equal to $\mu$. To see this, notice that if a sequence of polynomials $f_n\in T_{\mu,n}$ converges to a limit $f$, then $f$ can be factorized in the same pattern as each $f_n$ because  being a product is a closed condition. However, the irreducible factors of $f_n$ might become reducible in the limit because  being irreducible is an open condition. Hence, $\ld_f$ is finer than or equal to $\ld_\mu$. 

We will abbreviate the compactly supported cochain complex  $C^*_c(\rd;\Z)$ simply as $C^*$. The increasing filtration $\{\mathcal{F}_{p}\}$ of $\rd$ induces a decreasing filtration $\{\mathcal{G}^{p} C^*\}$ of $C^*$:
$$    C^*=\mathcal{G}^{0}C^*\supset\mathcal{G}^{1}C^*\supset\cdots\supset\mathcal{G}^{d-1}C^*=0 \ \ \ \ \text{ where each }\ \mathcal{G}^{p}C^*:=C^*_c(\rd\setminus \mathcal{F}_{p};\Z).
$$
Since each $\mathcal{F}_p$ is a closed subspace of $\rd$, we have $$\frac{\mathcal{G}^{p}C^*}{\mathcal{G}^{p+1}C^*}
\cong C_c^*(\mathcal{F}_{p+1}\setminus \mathcal{F}_{p};\Z).$$
 Thus, the filtered complex $\{\mathcal{G}^{p} C^*\}$ induces a spectral sequence with $E_1$-page:
$$E_1^{p,q}=H_c^{p+q}(\mathcal{F}_{p+1}\setminus \mathcal{F}_{p};\Z)\ \ \Longrightarrow\ \ \ H^{p+q}_c(\rd;\Z).$$
Finally, by (\ref{fil}) we have
\begin{equation}\label{disjoint}
    \mathcal{F}_{p+1}\setminus \mathcal{F}_{p} = \bigcup_{\ld:\ |\ld|= d-p}\s.
\end{equation}
For any two distinct partitions $\ld$ and $\mu$ of equal size $d-p$, we have $\overline{T_{\mu,n}}\cap \s=\emptyset$ because it is impossible that $\ld$ is finer than $\mu$. Thus, the set-theoretical disjoint union (\ref{disjoint}) is actually a disjoint union of topological spaces. Hence, we obtain the spectral sequence (\ref{spec}).

Notice that $E_1^{p,q}$ is nonzero only when $0\le p\le d-2$. Thus, all $E_r$-differentials are zero when $r\ge d-1$.
\end{proof}

\section{Proof of Theorem \ref{homo stab low}} \label{proof1}
Theorem \ref{homo stab low} in the Introduction will  follow from Theorem \ref{iso} below together with Lemma \ref{projective}.

\begin{thm}
\label{iso}
For $n>1$, the inclusion $\ird\hookrightarrow \pld$ induces an isomorphism
$$H_i(\ird,\Z)\xrightarrow{\cong} H_i(\pld,\Z)$$
when $i\le2\bigg[{{d+n-1}\choose {n-1}}-n-1$\bigg].
\end{thm}

\begin{rmk}
In \cite{Car}, Carlitz obtained his result by showing that $|\irdf|\sim|\mathrm{Poly}_{d,n}(\fq)|$ as $d\to\infty$ when $n>1$. Theorem \ref{iso} is a topological analog of Carlitz' observation that ``when the number of indeterminates is greater than one we find that almost all polynomials are irreducible" (\cite{Car}, Section 1). The assumption $n>1$ is needed in our proof below. 
\end{rmk}

\begin{proof}[Proof of Theorem \ref{iso}]
Since $\ird\hookrightarrow\pld$ is an inclusion of complex (hence orientable) manifolds of equal complex dimension $\Big[{{{d+n}\choose n}-1}\Big]$, by Poincar\'e duality, in order to prove Theorem \ref{iso}, it suffices to prove that the inclusion $\ird\hookrightarrow \pld$ induces an isomorphism on compactly supported cohomology
$$H^i_c(\ird,\Z)\xleftarrow{\cong} H^i_c(\pld,\Z)$$
when $i\ge 2\Big[{{{d+n}\choose n}-1}\Big]-2\Big[{{d+n-1}\choose {n-1}}-n-1\Big]=2[\binom{d+n-1}{n}+n]$. By the long exact sequence (\ref{les}), it suffices to prove the following proposition:

\begin{prop}\label{vanishing}
For $n>1$, we have $H^i_c(\rd;\Z)=0$ when $i\ge 2\Big[\binom{d+n-1}{n}+n\Big]-1$.
\end{prop}
Before proving Proposition \ref{vanishing}, we will first prove the following lemma:
\begin{lemma}\label{bound}
For any partition $\ld$ of $d$ such that $|\ld|\ge2$, we have $\dim_\CC(\s)\le\binom{d+n-1}{n}+n-1$.
\end{lemma}
\begin{proof}
$\ird$ is an open subset of $\pld$ and thus is a manifold of complex dimension $\Big[{{{d+n}\choose n}-1}\Big]$. Hence, $\sym^m(\ird)$ is a orbifold (\emph{i.e.} a manifold quotient by a finite group action) of complex dimension $m\Big[{{{d+n}\choose n}-1}\Big]$. By (\ref{sym}), each $\s$ is also an orbifold with dimension
\begin{equation}
\label{dim}
    \dim_\CC(\s)=\sum_{j=1}^d m_j(\ld) \bigg[\binom{j+n}{n}-1\bigg].
\end{equation}
Since the function $\binom{j+n}{n}-1$ is strictly convex in $j$ when $n>1$, we have $ \dim_\CC(\s)< \dim_\CC(T_{\mu,n})$ if $\ld$ is strictly finer than $\mu$. Therefore,  $\dim_\CC (\s)$ is maximized at some partition $\ld$ of size exactly 2. Hence, it suffices to consider $\ld$ to be of the form $k+(d-k)$ for some integer $k=1,\cdots,d-1$. For such $\ld$, we have
$$\dim_\CC(\s)=\binom{k+n}{n}+\binom{d-k+n}{n}-2=:f(k).$$
By checking its second derivative, the function $f(k)$ is strictly convex for $k\in[1,d-1]$ and thus the only possible local maximum occur at the two endpoints. Hence, for any $k\in[1,d-1]$, we have $f(k)\le f(1)=f(d-1)$.
\end{proof}

\begin{proof}[Proof of Proposition \ref{vanishing}]
Consider the spectral sequence in Lemma \ref{ss}:
$$  E_1^{p,q} = \bigoplus_{\ld:\ \ld\vdash d, |\ld|=d-p\ge2} H_c^{p+q}(\s;\Z)\ \Longrightarrow\ H^{p+q}_c(\rd;\Z).$$
Lemma \ref{bound} implies that $E_1^{p,q}=0$ when $p+q>2\Big[\binom{d+n-1}{n}+n-1\Big]$. Thus, we have
$$H^{i}_c(\rd;\Z)=0$$
 when $i>2\Big[\binom{d+n-1}{n}+n-1\Big]$.
\end{proof}
Theorem \ref{iso} now follows from Proposition \ref{vanishing}.
\end{proof}
%
%
%
%

\section{Proof of Theorem \ref{high stab}}

\subsection{Preliminary results} We first obtain some preliminary results to be used later in the proof.

\begin{lemma}\label{red ird}
For any $n$ and $d>1$, and for any $i$ in the range as stated in Theorem \ref{high stab}, the natural connecting homomorphism is an isomorphism:
\begin{equation}\label{delta}
    H^{i}_c(\rd;\QQ)\xrightarrow{\cong} H^{i+1}_c(\ird;\QQ).
\end{equation}
\end{lemma}
\begin{proof}
Again as in the proof of Theorem \ref{iso} above, the decomposition $\ird = \pld\setminus\rd$ gives the following long exact sequence
$$\cdots\to H^{i}_c(\ird;\QQ)\to H^i_c(\pld;\QQ)\to H^{i}_c(\rd;\QQ)\to \cdots$$
By Lemma \ref{projective}, we have $H^i_c(\pld;\QQ)=0$ when $i<2{n+d-1\choose n}$. Thus, (\ref{delta}) is an isomorphism when $i<2{n+d-1\choose n}-1$. Finally, we need to check that the reasoning above holds in the range of $i$ stated in Theorem \ref{high stab}, which is equivalent to checking that 
$$\frac{2n}{d-1}-1-\frac{(d-2)(d-3)}{2}\le2{n+d-1\choose n}-1.$$
Indeed, when $d>1$, we have
$$\text{LHS}\le \frac{2n}{d-1}-1<\frac{2(n+d-1)}{d-1}-1\le 2\cdot \frac{n+d-1}{d-1}\frac{n+d-2}{d-2}\cdots \frac{n}{1}-1=\text{RHS}.$$
\end{proof}

Next, we prove the following general results about graded vector spaces. 
\begin{lemma}\label{graded vs}
Suppose $f:A\to B$ and $g:C\to D$ are maps of graded vector spaces.  If for any $i\le r$, the maps $f:A_i\xrightarrow[]{\cong} B_i$ and $g:C_i\xrightarrow[]{\cong} D_i$ are isomorphisms on the $i$-th graded pieces, then the following maps
$$ (A\otimes C)_i\xrightarrow[]{f\otimes g} (B\otimes D)_i$$
$$ (A^{\otimes m})^{{S}_m}_i\longrightarrow(B^{\otimes m})^{{S}_m}_i$$
are also isomorphisms for any $i\le r$.
\end{lemma}
\begin{proof}
For any $i\le r$, we have
\begin{align*}
    (A\otimes C)_i&=\bigoplus_{ s+t=i}A_s\otimes C_t\xrightarrow[\cong]{f\otimes g}\bigoplus_{s+t=i}B_s\otimes D_t= (B\otimes D)_i
\end{align*}
since each $s$ and $t$ in the summand are no more than $i$ and hence $r$. Moreover, applying the reasoning above inductively on $m$, we have 
$$(A^{\otimes m})_i\xrightarrow[\cong]{f^{\otimes m}} (B^{\otimes m})_i\ \ \ \ \ \text{for }i\le r.$$
Observe that the isomorphism is equivariant with respect to the action of ${S}_m$. Taking the ${S}_m$-invariants, we obtain the second claim.
\end{proof}

Finally,  we apply Lemma \ref{graded vs} to study the compactly supported cohomology of symmetric powers.
\begin{lemma}
\label{sym coh}
Suppose $X$ is a closed subspace of $Y$ such that the inclusion $inc:X\hookrightarrow Y$ induces an isomorphism
$$inc^*:H^i_c(Y;\QQ)\xrightarrow{\cong} H^i_c(X;\QQ)$$
for any $i\le r$. Then for any natural number $m$, the inclusion $inc: \sym^m(X)\hookrightarrow \sym^m(Y)$ also induces an isomorphism
$$inc^*:H^i_c(\sym^m(Y);\QQ)\xrightarrow{\cong} H^i_c(\sym^m(X);\QQ)$$
for any $i\le r$. 
\end{lemma}

\begin{proof}
Since $X$ is a closed subspace of $Y$, the symmetric power $\sym^m(X)$ is also a closed subspace of $\sym^m(Y)$. Hence the inclusion map $\sym^m(X)\hookrightarrow \sym^m(Y)$ is proper and induces maps on cohomology groups with compact support. 

Moreover, we have
$$H^*_c(\sym^m Y;\QQ) = H^*_c(Y^m/{S}_m;\QQ)\cong H^*_c(Y^m ;\QQ)^{{S}_m}\cong (H^*_c(Y;\QQ)^{\otimes m})^{{S}_m}$$
where the second isomorphism is the transfer homomorphism. Lemma \ref{sym coh} now follows by applying Lemma \ref{graded vs} to $A=H^*_c(\sym^m Y;\QQ)$ and $B=H^*_c(\sym^m X;\QQ)$.
\end{proof}

\subsection{The proof of Theorem \ref{high stab}}
 
We proceed by induction on $d\ge 2$. First we check the base case when $d=2$. We have
$$\red_{2,n}(\CC)= \sym^2\pl_{1,n}(\CC) \cong \sym^2\CC P^{n-1}.$$
The inclusion $\CC P^{n-1}\hookrightarrow \CC P^n$ induces an isomorphism on $i$-th rational compactly supported cohomology when $i\le 2n-2$. Thus, by Lemma \ref{sym coh}, we have 
$$H^i_c(\red_{2,n}(\CC);\QQ)\cong H^i_c(\red_{2,n+1}(\CC);\QQ) $$
when $i\le2n-2$. By Lemma \ref{red ird}, we have 
$$H^i_c(\ir_{2,n}(\CC);\QQ)\cong H^i_c(\ir_{2,n+1}(\CC);\QQ) $$
when $0<i\le 2n-1.$ The isomorphism also holds when $i=0$ because $\ir_{2,n}(\CC)$ and $\ir_{2,n+1}(\CC)$ are both connected (by Theorem \ref{homo stab low}) and noncompact and thus both have $H_c^0=0$.

For induction, suppose that for a fixed $d>1$, our claim is true for any $j<d$. We want to prove the claim for $d$.  Again, since $\ird$ is connected and noncompact, it has vanishing $H^{0}_c$. Theorem \ref{high stab} is already true for $i=0$. By Lemma \ref{red ird}, it suffices to prove the following claim for our fixed $d$.

\begin{claim}
\label{red stab}
For any $n>1$ and for any $i<\frac{2n}{d-1}-\frac{(d-2)(d-3)}{2}-2$, the inclusion $\rd \hookrightarrow\rdd$ induces an isomorphism
$$H^i_c(\rd; \QQ)\xleftarrow[]{\cong} H^i_c(\rdd;\QQ).$$
\end{claim}

\begin{proof}

We will prove Claim \ref{red stab} in two steps: first, we show that it will follow from Claim \ref{ld stab} below, and second, we show that Claim \ref{ld stab} follows from the induction hypothesis.

\noindent\textbf{Step 1.} We first reduce Claim \ref{red stab} to the following claim:
\begin{claim}
\label{ld stab}
For any $n>1$, for any non-singleton partition $\ld\vdash d$ , and for any $i<\frac{2n}{d-1}-\frac{(d-2)(d-3)}{2}+d-4$, the inclusion $\s \hookrightarrow T_{\ld,n+1}$ induces an isomorphism
$$H^i_c(\s; \QQ)\xleftarrow[]{\cong} H^i_c(T_{\ld,n+1};\QQ).$$
\end{claim}
\begin{proof}[Proof of that Claim \ref{ld stab} $\Rightarrow$ Claim \ref{red stab}]

Consider the spectral sequence in Lemma \ref{ss} tensored with $\QQ$:
\begin{equation}
\label{ss1}
    E_1^{p,q} = \bigoplus_{\ld\vdash d, |\ld|=d-p\ge2} H_c^{p+q}(\s;\QQ)\ \Longrightarrow\ H^{p+q}_c(\rd;\QQ)
\end{equation}
Consider the same spectral sequence for $\red_{d,n+1}$:
\begin{equation}
\label{ss2}
    F_1^{p,q} = \bigoplus_{\ld\vdash d, |\ld|=d-p\ge2} H_c^{p+q}(\sn;\QQ)\ \Longrightarrow\ H^{p+q}_c(\rdn;\QQ)
\end{equation}
Since the inclusion $\rd\hookrightarrow\rdn$ preserves the filtration (\ref{fil}), it induces a map between the two spectral sequences (\ref{ss1}) and (\ref{ss2}). 

Claim \ref{ld stab} implies that the inclusion $\rd\hookrightarrow\rdn$ induces an isomorphism between the 1st pages of the spectral sequences (\ref{ss1}) and (\ref{ss2})
$$E_1^{p,q}\xleftarrow[]{\cong}F_1^{p,q},\ \ \text{ when } p+q<\frac{2n}{d-1}-\frac{(d-2)(d-3)}{2}+d-4.$$
Taking the next page, we have
$$E_2^{p,q}\xleftarrow[]{\cong}F_2^{p,q},\ \ \text{ when } p+q<\frac{2n}{d-1}-\frac{(d-2)(d-3)}{2}+d-5.$$
In general, we have
$$E_r^{p,q}\xleftarrow[]{\cong}F_r^{p,q},\ \ \text{ when } p+q<\frac{2n}{d-1}-\frac{(d-2)(d-3)}{2}+d-3-r.$$
By Lemma \ref{ss}, the spectral sequences $E$ and $F$ both converge at page $d-1$. Thus we have
$$H^i_c(\rd; \QQ)\xleftarrow[]{\cong} H^i_c(\rdd;\QQ)$$
when $i<\frac{2n}{d-1}-\frac{(d-2)(d-3)}{2}-2.$ 
\end{proof}

\noindent\textbf{Step 2.} Finally, we will prove Claim \ref{ld stab} assuming our induction hypothesis: 
for any $j<d$, for any $n>1$, the natural inclusion $\ir_{j,n}(\CC) \hookrightarrow\ir_{j,n+1}(\CC)$ induces an isomorphism
$$H^i_c(\ir_{j,n}(\CC); \QQ)\xleftarrow[]{\cong} H^i_c(\ir_{j,n+1}(\CC);\QQ)$$
when $i<\frac{2n}{j-1}-\frac{(j-2)(j-3)}{2}-1$.


To prove Claim \ref{ld stab}, we first notice that since $\ld\vdash d$ is a non-singleton partition, each part of $\ld$ must have length at most $d-1$. Compare the following two isomorphisms of graded vector spaces given by (\ref{sym}):
\begin{align}
     H^*_c(\s;\QQ)&\cong\bigotimes_{j= 1}^{d-1}H^*_c(\mathrm{Sym}^{m_j(\ld)}\ir_{j,n}(\CC);\QQ)\label{1}\\
     H^*_c(T_{\ld,n+1};\QQ)&\cong\bigotimes_{j= 1}^{d-1}H^*_c(\mathrm{Sym}^{m_j(\ld)}\ir_{j,n+1}(\CC);\QQ)\label{2}
\end{align}
Let $r:=\frac{2n}{d-1}-\frac{(d-2)(d-3)}{2}+d-4$, which is the upper bound for $i$ in Claim \ref{ld stab}. We claim that 
\begin{equation}
    \label{e}
    r\le \frac{2n}{j-1}-\frac{(j-2)(j-3)}{2}-1\ \text{ for any } j<d.
\end{equation}
Notice that the right hand side is non-increasing in $j$ taken integer values. So it suffices to check the inequality (\ref{e}) for $j=d-1$. We calculate that
$$    \bigg(\text{RHS of (\ref{e}) evaluated at } j=d-1\bigg)-r = \frac{2n}{d-2}-\frac{2n}{d-1}>0
$$
confirming the inequality (\ref{e}). Hence, for any $i<r$ as in the assumption of Claim \ref{ld stab}, we must also have $i<\frac{2n}{j-1}-\frac{(j-2)(j-3)}{2}-1$ and hence by the induction hypothesis we have
$$H^i_c(\ir_{j,n}(\CC); \QQ)\xleftarrow[]{\cong} H^i_c(\ir_{j,n+1}(\CC);\QQ).$$
Thus, if we compare (\ref{1}) and (\ref{2}) using  Lemma \ref{graded vs} and Lemma \ref{sym coh}, we have that for any $i<r$
$$H^*_c(\s;\QQ)\xleftarrow[]{\cong} H^*_c(T_{\ld,n+1};\QQ)$$
which gives Claim \ref{ld stab}. As a consequence, Claim \ref{red stab} follows. By induction, we obtain Theorem \ref{high stab}.
\end{proof}

\section{A vanishing theorem}




\begin{thm}
\label{vanishing2}
For $d,n>1$, when $k\le 2d$, we have
$$H^k_c(\ird;\QQ)=0.$$
\end{thm}
Taking the limit $n\to\infty$, we obtain the following corollary.
\begin{cor}\label{van limit}
For $d>1$, when $k\le 2d$, we have 
$$b_k(d):=\lim_{n\to\infty}\dim H^k_c(\ird;\QQ)=0.$$
\end{cor}

\begin{proof}[Proof of Theorem \ref{vanishing2}]
We will prove Theorem \ref{vanishing2} in three steps.

\noindent\textbf{Step 1.} We  first collect some general results about graded vector spaces.

\begin{lemma}\label{graded 0}
Suppose $A$  and $B$ are graded vector spaces. If $A_i=0$ for any $i<a$, and $B_j=0$ for any $j<b$, then
\begin{enumerate}[label=(\roman*)]
    \item $(A\otimes B)_k=0$ for any $k<a+b$,
    \item $(A^{\otimes m})^{{S}_m}_k=0$ for any $k<ma$.
\end{enumerate}
\end{lemma}
\begin{proof}
Suppose (i) is false: there exists some $k< a+b$ such that  $(A\otimes B)_k=\bigoplus_{i+j=k} A_i\otimes B_j\ne0$.  There exist some $i,j$ such that $i+j=k$ and $A_i\ne0$ and $B_j\ne 0$, which implies that $i\ge a$ and $j\ge b$, and thus $i+j\ge a+b$, contradicting the assumption that $i+j=k<a+b$.

Applying (i) inductively on $m$, we obtain that $(A^{\otimes m})_k=0$ for any $k< ma$. Thus the ${S}_m$-invariant subspace must also be zero.\end{proof}

\noindent\textbf{Step 2.} We will inductively define a function $r:\NN\to\NN$ and compute its value. 

\begin{defn}
For each positive integer $d$, define $r(d)$ inductively by
\begin{align}
    &r(1)=2\nonumber\\
    \forall d>1,\ \ \ \  &r(d)=1+\min\Big\{ r(\ld): \ld\vdash d, |\ld|\ge 2\Big\},\label{r}\\
    &\text{ \ \ \ \ where for each $\ld\vdash d$ such that $|\ld|\ge 2$, we define } r(\ld):=\sum_{j=1}^{d-1} m_j(\ld) r(j).\nonumber
\end{align}
\end{defn}
\begin{prop}\label{r2}
Suppose that  $d>1$.
\begin{enumerate}[label=(\alph*)]
    \item For any $\ld\vdash d$ such that $|\ld|\ge 2$, we have $r(\ld)=2d+|\ld|-m_1(\ld)$.
    \item The minimum $r(\ld)$ is uniquely achieved at $\ld=\hat{1}$  where $\hat{1}$ stands for the partition $d=1+\cdots+1$.
    \item $r(d)=2d+1$.
\end{enumerate}
\end{prop}
\begin{proof}
 We will prove the three statements by induction on  $d\ge2$. For the base case when $d=2$,  the three statements are easily verified since $1+1$ is the only non-singleton partition of $2$. 
 
 For induction, we consider the case when  $d>2$, assuming the three statements all hold for any $e<d$. For any non-singleton partition $\ld\vdash d$ with $|\ld|\ge 2$, we have
\begin{align*}
    r(\ld)&:= \sum_{j=1}^{d-1} m_j(\ld) r(j)\\
    &= m_1(\ld)\cdot 2+\sum_{j=2}^{d-1} m_j(\ld) (2j+1)&\text{by induction hypothesis (3)}\\
    &=  \underbrace{\sum_{j=1}^{d-1} m_j(\ld)\cdot 2j}_{=2d}\ +\ \underbrace{\sum_{j=2}^{d-1} m_j(\ld)}_{=|\ld|-m_1(\ld)} \\
    &= 2d+|\ld|-m_1(\ld)
\end{align*}
Thus, (a) is verified. (b) and (c) follow immediately from (a).
\end{proof}

\noindent\textbf{Step 3.} We will prove the following vanishing result in a range defined by the function $r$. 
\begin{prop}\label{van r}
 For any $d\ge1$ and $n\ge 2$,
\begin{enumerate}
\item for any  partition $\ld\vdash d$ such that $|\ld|\ge2$, we have
    $$H^k_c(\s;\QQ)=0\ \ \ \ \ \ \ \text{    when $k<r(\ld)$
}$$ 
\item     $$H^k_c(\ird;\QQ)=0\ \ \ \ \ \ \ \text{    when $k<r(d)$.
}$$
\end{enumerate}
\end{prop}
\begin{proof}
We will prove both statements by induction on $d$. For the base case when $d=1$, part (1)  is vacuously true. We have 
$$\ir_{1,n}(\CC)= \pl_{1,n}(\CC) \cong \CC\times \CC P^{n-1}$$
where the second homeomorphism comes from  Lemma \ref{projective}. Part (2) is also verified.

We consider the case when $d\ge 2$ for the induction, assuming both (1) and (2) hold for any $e<d$. Since $\ld$ has at least two parts, each part has size at most $d-1$. Recall that (\ref{sym}) gives:
$$    \s\cong\prod_{j=1}^{d-1}\mathrm{Sym}^{m_j(\ld)}\Big(\ir_{j,n}(\CC)\Big).$$

By induction hypothesis part (2), for each $j\le d-1$, we have that
$H^k_c(\ir_{j,n}(\CC);\QQ)=0$ when $k<r(j)$. Now we briefly digress to prove the following general results about symmetric powers.
\begin{lemma}\label{sym 0}
If $H^k_c(Y;\QQ)=0$ for any $k<r$, then $H^k_c(\sym^m Y;\QQ)=0$ for any $k<mr$.
\end{lemma}
\begin{proof}
Observe that 
$$H^*_c(\sym^m Y;\QQ) =H^*_c(Y^{\times m}/{S}_m;\QQ) \cong H^*_c(Y^{\times m};\QQ)^{{S}_m} \cong (H^*_c(Y;\QQ)^{\otimes m})^{{S}_m}.$$
Apply part (ii) of Lemma \ref{graded 0}.
\end{proof}
Thus,  by Lemma  \ref{sym 0}, we have that $H^k_c\Big(\sym^{m_j(\ld)}\ir_{j,n}(\CC);\QQ\Big)=0$ when $k<m_j(\ld)r(j)$. By Lemma \ref{graded 0} part (i), we have 
\begin{equation}\label{l vanishing}
    H^k_c(\s;\QQ)=0
\end{equation}
when $k<\sum_{j=1}^{d-1}m_j(\ld)r(j)=r(\ld)$. Part (1) is verified.

To prove part (2), we consider the spectral sequence (\ref{spec})
$$  E_1^{p,q} = \bigoplus_{\ld:\ \ld\vdash d, |\ld|=d-p\ge2} H_c^{p+q}(\s;\QQ)\ \Longrightarrow\ H^{p+q}_c(\rd;\QQ)$$
By (\ref{l vanishing}), we know that $E_1^{p,q}=0$ in the range when $p+q<r(d)-1$. Thus, when $k<r(d)-1$,
$$H^{k}_c(\rd;\QQ)=0.$$
Since $n,d\ge 2$, by Proposition \ref{r2}, we have $r(d)-1=2d< 2{n+d-1\choose n}-1$. Thus, by Lemma \ref{red ird}, when $i<r(d)-1$, we have
$$H^{i+1}_c(\ird;\QQ)\cong H^{i}_c(\rd;\QQ)=0.$$
Part (2) is verified.
\end{proof}
Finally, combining  part (2) of Proposition \ref{van r} and part (c) of Proposition \ref{r2}, we obtain Theorem \ref{vanishing2}.
\end{proof}

\section{Appendix: Computation for $d\le4$}
In this appendix, we consider the stable cohomology in Theorem \ref{high stab}, more precisely, the limit
\begin{equation}
    \label{b_i}
    b_i(d):=\lim_{n\to\infty} \dim H^i_c(\ird;\QQ)
\end{equation}
for $d\le 4$. Theorem \ref{high stab} tells us that the limit exists. The purpose of our computations here is to illustrate that the stable cohomology in Theorem \ref{high stab} are generally nonzero despite the vanishing result in Theorem \ref{vanishing2}, and  that the spectral sequence (\ref{spec}) which is central in the previous proofs has nontrivial differentials, even in the stable range. To keep this appendix brief, we will only sketch the computations, highlighting the analysis of differentials in the spectral sequence.

As in Theorem \ref{high stab}, there is a closed embedding (hence a  proper map) $\pld\to\pl_{d,n+1}(\CC)$ for each $n$. We define the following direct limits
\begin{align*}
    &\pl_d(\CC):=\lim_{\longrightarrow} \pld\\   &\ir_d(\CC):=\lim_{\longrightarrow} \ird\\
    &\red_d(\CC):=\lim_{\longrightarrow} \rd\\
    &\mathrm{T}_{\ld}:=\lim_{\longrightarrow} \s \ \ \  \text{ for each $\ld\vdash d$}.
\end{align*}
Since compactly supported cohomology preserves limits, the stable cohomology can be expressed as:
$$b_i(d)=\dim H^i_c(\ir_d(\CC);\QQ).$$
All cohomology considered in this section will be over $\QQ$. We will therefore suppress  the $\QQ$-coefficients from our notation. We will encode our computation of $H^i_c(\ii)$ into a Poincar\'e series:
$$P_d(t):=\sum_{i} b_i(d) t^i.$$

\noindent\textbf{d=1.} We have $\ir_{1,n}(\CC)=\pl_{1,n}(\CC) = \CC\times \CC P^{n-1}$ by Lemma \ref{projective}. Thus, as $n\to\infty$, we have 
$$P_1(t)= \frac{t^2}{1-t^2}.$$

\noindent\textbf{d=2.} 
By Lemma \ref{red ird}, when $d>1$ and $n\to\infty$, we have that for every $i$,
\begin{equation}\label{r vs i}
    H^i_c(\red_d(\CC))\cong H^{i+1}_c(\ii).
\end{equation}
For $V$ a graded vector space, we use $s^k V$ to denote $V$ with grading shifted by $k$ (\emph{a.k.a} the $k$-th suspension of $V$) where $(s^k V)_i = V_{i-k}$. When $d=2$, we have 
\begin{align*}
    H^*_c(\red_2(\CC))&=H^*_c\Big(\sym^2 \big(\ir_1(\CC)\big)\Big)\\
    &\cong\sym^2 \big( H^*_c(\ir_1(\CC))\big)\\
    &\cong s^4\sym^2 H^*(\CC P^\infty)\\
    &\cong s^4 \QQ[e_1,e_2]\ \ \ \ \ \ \ \text{ where } |e_1|=2, |e_2|=4
\end{align*}
The last isomorphism comes from the fundamental theorem of symmetric polynomials. Thus, we conclude
\begin{equation}\label{d=2}
    P_2(t)=\frac{t^5}{(1-t^2)(1-t^4)}.
\end{equation}

\noindent\textbf{d=3.} There are two non-singleton partitions of $d=3$, namely $3=1+1+1$ and $3=1+2$. Let $T_{1+1+1}$ denote the stratum corresponding to the partition $1+1+1=3$, and so on. We have $\red_3(\CC)=T_{1+2}\cup T_{1+1+1}$ where $T_{1+1+1}$ is closed. The associated long exact sequence gives the following connecting homomorphism:
$$\delta_i:H^i_c(\sym^3\ir_{1}(\CC))\to H^{i+1}_c(\ir_2(\CC)\times \ir_1(\CC)).$$
We now show that the differential $\delta_i$ must be injective for every $i$. There is a surjective map $\pl_2(\CC)\times \ir_1(\CC)\to \red_3(\CC)$, given by the multiplication of two polynomials. The preimage of the closed subspace $T_{1+1+1}$ is $T_{1+1}\times T_1$, while the preimage of the open subspace $T_{1+2}$ is $T_2\times T_1$. We obtain the following commutative diagram:
\begin{equation}\label{inj}
\begin{tikzpicture}
\diagram (m)
{ H^i_c(\sym^3\ir_{1}(\CC))& H^{i+1}_c(\ir_2(\CC)\times \ir_1(\CC)) \\
   H^i_c(\sym^2\ir_{1}(\CC)\times \ir_1(\CC)) & H^{i+1}_c(\ir_2(\CC)\times \ir_1(\CC)) \\};
\path [->] (m-1-1) edge node [above] {$\delta_i$} (m-1-2)
                   edge node [left] {transfer} (m-2-1)
           (m-2-1) edge node [above] {$\delta_i^\prime\otimes id$} (m-2-2)
           (m-2-2) edge node [right] {$=$} (m-1-2);
\end{tikzpicture}
\end{equation}
The vertical map is a transfer homomorphism, given by including the ${S}_3$-invariant subspace of $H^i_c(\ir_{1}(\CC))^{\otimes 3}$ into the ${S}_2\times {S}_1$-invariant subspace. The differential $\delta_i^\prime:  H^i_c(\sym^2\ir_{1}(\CC))\to H^{i+1}_c(\ir_2(\CC) $ is an isomorphism for all $i$ by (\ref{r vs i}). Hence, $\delta$ must be injective, which implies
\begin{align*}
    H^{*+1}_c(\red_3(\CC))&\cong coker(\delta_i)\\
    &\cong \frac{\text{$S_2\times S_1$-invariant subspace of $H^*_c(\ir_1(\CC))^{\otimes 3}$}}{\text{$S_3$-invariant subspace of $H^*_c(\ir_1(\CC))^{\otimes 3}$}}
\end{align*}
We calculate the Poincar\'e series of the numerator and the denominator in the same way as in (\ref{d=2}), taking their difference and multiply by an appropriate power of $t$ to account for the degree shift, and obtain 
$$P_3(t)=\frac{t^{10}}{(1-t^2)(1-t^6)}.$$

\noindent\textbf{d=4.} A calculation of $P_4(t)$ is already too complex for us to sketch here in any reasonable length. Instead, we will be content with finding the first nonzero stable cohomology. We will show that $b_i(4)=0$ for any $i<11$ and that $b_{11}(4)=1$. Hence, $P_4(t)=t^{11}+O(t^{12})$.

There are four non-singleton partitions of 4. The partition poset is ordered below:

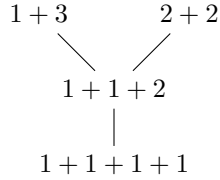
\begin{center}
    \begin{tikzpicture}
  \node (a) at (-1,1) {$1+3$};
  \node (c) at (1,1) {$2+2$};
  \node (e) at (0,0) {$1+1+2$};
  \node (min) at (0,-1) {$1+1+1+1$};
  \draw  (e) -- (min);
  \draw[preaction={draw=white, -,line width=6pt}] (a) -- (e) -- (c);
\end{tikzpicture}
\end{center}

The three levels of the partition lattice above induce a spectral sequence (\ref{spec}) with three columns. All terms in the spectral sequence with total degree $\le 8$ must be zero, by our previous computations for $d\le3$. Below we draw the region of the spectral sequence  with total degree $\le 10$. The column $p=0$ comes from $H^*_c(T_{1+1+1+1})$ where $H^{8}_c\cong H^{10}_c\cong \QQ$. The column $p=1$ comes from $H^*_c(T_{2+1+1})$ where $H^9_c\cong H^{11}_c\cong \QQ.$ The column $p=2$ comes from $H^*_c(T_{3+1})\oplus H^*_c(T_{2+2})$ where $H^{10}_c(T_{2+2})\cong \QQ$ by the case $d=2$ above, and $H^{10}_c(T_{1+3})=0$ by the cases $d=1$ and $d=3$ above.

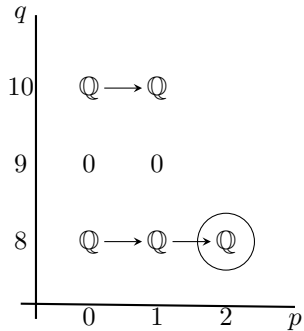
\begin{center}
    
\begin{tikzpicture}
  \matrix (m) [matrix of math nodes,
    nodes in empty cells,nodes={minimum width=5ex,
    minimum height=5ex,outer sep=-5pt},
    column sep=1ex,row sep=1ex]{
          q      &      &     &     & \\
          10     &  \QQ & \QQ  &  & \\
          9     &  0 &  0  &  & \\
          8     &  \QQ  & \QQ &  \ZZZ  & \\
    \quad\strut &   0  &  1  &  2  & p& \strut \\};
    \draw[-stealth] (m-4-2.east) -- (m-4-3.west);
        \draw[-stealth] (m-4-3.east) -- (m-4-4.west);
            \draw[-stealth] (m-2-2.east) -- (m-2-3.west);
\draw[thick] (m-1-1.east) -- (m-5-1.east) ;
\draw[thick] (m-5-1.north) -- (m-5-5.north) ;
\end{tikzpicture}
\end{center}
By the same argument as in (\ref{inj}), the two differentials from the first column $p=0$ to the second column $p=1$ in the diagram above must be injective. Consequently, the differential $\delta: E_1^{1,8}\to E_1^{2,8}$ must be zero. Hence, $E_1^{2,8}$, circled in the diagram, must survive in the $E_\infty$-page, contributing to $H^{10}_c(\red_4(\CC))\cong H^{11}_c(\ir_4(\CC))\cong \QQ$.

\end{document}